\documentclass{article}
\usepackage[latin1]{inputenc}
\usepackage[T1]{fontenc}
\usepackage[british,UKenglish,USenglish,american]{babel}
\usepackage{amsmath}
\usepackage{amsfonts}
\usepackage{amssymb}
\usepackage{hyperref}
\usepackage{mathrsfs}
\usepackage{makeidx}
\usepackage{fancyhdr}
\usepackage{verbatim}
\usepackage{bbm}
\usepackage{bm}
\usepackage{mathtools}

\usepackage{graphics}
\usepackage{color}
\usepackage{booktabs}
\usepackage{ulem}

\numberwithin{equation}{section}
\usepackage[titletoc,title]{appendix}
\usepackage{amsthm}
\usepackage{cleveref} 
\usepackage{todonotes}
\numberwithin{equation}{section}
\newcommand{\be}{\begin{eqnarray}}
\newcommand{\ee}{\end{eqnarray}}
\newcommand{\ce}{\begin{eqnarray*}}
\newcommand{\de}{\end{eqnarray*}}
\newtheorem{theorem}{Theorem}[section]
\newtheorem{lemma}[theorem]{Lemma}
\newtheorem{proposition}[theorem]{Proposition}
\newtheorem{corollary}[theorem]{Corollary}
\theoremstyle{remark}

\newtheorem{example}[theorem]{Example}
\newtheorem{remark}[theorem]{Remark}
\newtheorem{definition}[theorem]{Definition}

\Crefname{eqn}{Equation}{Equations}
\Crefname{assumption}{Assumption}{Assumptions}
\Crefname{innercustomthm}{Condition}{Conditions}
\crefrangelabelformat{innercustomthm}{#3#1#4-#5#2#6}

\def\R{\mathbb{R}}

\def\N{\mathbb{N}}
\def\Z{\mathbb{Z}}

\def\T{\mathbb{T}}

\def\cA{\mathcal{A}}
\def\cB{\mathcal{B}}

\def\cE{\mathcal{E}}
\def\cF{\mathcal{F}}

\def\cH{\mathcal{H}}

\def\cP{\mathcal{P}}

\def\cS{\mathcal{S}}

\def\cZ{\mathcal{Z}}
\def\bd{\begin{definition}}
\def\ed{\end{definition}}
\def\bp{\begin{proposition}}
\def\ep{\end{proposition}}
\def\bc{\begin{corollary}}
\def\ec{\end{corollary}}
\def\bx{\begin{example}}
\def\ex{\end{example}}
\def\bG{\mathbf{G}}

\def\mC{{\mathbb C}}

\def\mE{{\mathbb E}}

\def\mI{{\mathbb I}}

\def\mN{{\mathbb N}}

\def\mP{{\mathbb P}}
\def\mQ{{\mathbb Q}}
\def\mR{{\mathbb R}}

\def\mT{{\mathbb T}}

\def\mZ{{\mathbb Z}}

\def\sF{{\mathscr F}}

\def\geq{\geqslant}
\def\leq{\leqslant}

\allowdisplaybreaks

\title{Numerical approximation of nonlinear fourth-order SPDEs with additive space-time white noise}
\author{Dirk Bl\"omker, Chengcheng Ling, Johannes Rimmele }
\date{\today}

\begin{document}

\maketitle

\begin{abstract}
   We consider the strong numerical approximation for a fourth-order stochastic nonlinear  SPDE driven by  space-time white noise on  {{$d=1,2,3$}}-dimensional torus.  
   We consider its full discretisation with 
   a spectral Galerkin scheme in space and Euler scheme in time. We show the convergence with almost spatial rate $2-\frac{d}{2}$ and  $\frac{6-d}{4}\wedge1$-temporal rate obtained mainly via \it{Stochastic Sewing} technique.
     \\

\end{abstract}

\section{Introduction}

\noindent
As one of the  typical fourth-order stochastic nonlinear  stochastic partial differential equations  (SPDEs), the generalized Cahn-Hilliard equation originated in the seminal works of Cahn and Hilliard on phase separation in binary alloys, where they introduced a free energy functional accounting for both bulk and interfacial contributions \cite{C, CH}.
In its deterministic form, this fourth-order partial differential equation models how a concentration field evolves to form distinct regions (phases) over time. The addition of random fluctuations or thermal noise gives rise to the (generalized) stochastic Cahn-Hilliard equation, providing a more accurate portrayal of microscopic uncertainties present in real materials 
\cite{D, DPG, BYZ, RYZ}.\\

As one of very active fields for research  \cite{C, Gy}, there has been a long list of work on numerical approximation of various types of SPDEs. \cite{BJ,BJ1} consider stochastic Burgers and \cite{BG, BDG, GS, BGJK, MZ} study stochastic Allen-Cahn equations, moreover \cite{J, JKW, JP, LQ} provide broader framework. For further interest on this topic, we refer to the references within the aforementioned works.

Stochastic Cahn-Hilliard  equations extend classical Cahn-Hilliard models with noise, requiring careful numerical methods. Early simulations adapted an unconditional gradient stable scheme \cite{E}, focusing on stable mass conservation \cite{D}. Gradually, finite-element, spectral, and hybrid strategies address various geometries and stability needs
\cite{KP,PZ}, driving ongoing improvements in efficiency and consistency.\\

The focus of this paper is the numerical approximation of a  generalized
Cahn-Hilliard type nonlinear SPDE on a $2$-dimensional torus, defined as:
\begin{align}
        \label{eq:SPDE}
        \partial_tu=-\delta\Delta^2u-\bG(u)+\sigma\xi 
    \end{align}
    where  $\bG:\mR \to\mR$ is the nonlinear term satisfying   $\|\bG\|_\infty, \|\partial \bG\|_\infty<\infty$ and  $\xi$ is the space-time white noise representing random fluctuations, scaled by a constant diffusion coefficient $\sigma>0$, which governs the intensity of the stochastic effects. We derive it as $\xi:=\partial_t W$, where $W$ is a standard cylindrical Wiener process defined in a filtered probability space  $(\Omega,\mathcal{F}, (\mathcal{F}_t)_{t\geq0},\mP)$.  We further propose the property of { $u$} to be in a moving frame, i.e. $\int_{\mT^d}u(t,x)dx=0$ for any $t\geq0$  {which represents the elimination of the zero Fourier mode}. The scheme we consider is the full discretisation of the equation which was first addressed by \cite{Gy}, more precisely (see also \eqref{eq:sol-n-N}), a finite difference in space and an explicit Euler
method in time based on sampling rectangular increments of $W$ on a grid with
meshsize $n^{-1}$
in time and $N^{-1}$
in space. \\

Inspired by the work \cite{BDG, DKP, DLR}, in the end we obtain the $L^p$-rate of convergence of order almost $N^{\frac{d}{2}-2}+n^{-\frac{6-d}{4}\wedge1}$. Evidently, we also overcome the order barrier $\frac{1}{4} $ with respect to the temporal step size which has been  addressed in \cite{JK, DKP}. The crucial idea within our analysis is to combine
the {\it regularity estimates on the semigroup generated by $-\Delta^2$} (see \Cref{lem:auxi-semi-U} and \Cref{lem:auxi-semigroup}) together with {\it Stochastic Sewing Lemma}  \cite{Le, DKP}  (see \Cref{lem:SSL}) 
so that we can  tune the spatial and temporal regularity of the mild solution and noise, which in the end yields the desired quantitative rate. This idea was originally invented to study the numerical approximation for singular SDEs \cite{BDG1, LL, GLL}. In this instance, we are able to apply it to generalized stochastic Cahn-Hilliard equations \eqref{eq:SPDE}. 
Let us also mention that, 
 in a future work \cite{DLR} we consider a surface growth model where nonlinear term has the form $\bG(u)=\nabla\cdot\frac{\nabla u}{1+|\nabla u|^2}$. 

\subsubsection*{Organization of the paper}
We introduce the definitions,  frequently used notations and main results in \Cref{sec:setting}. \Cref{sec:tools} presents  the  elementary  tools and estimates.  The proof of the main results is given in \Cref{sec:main-proof}.
\section{Preliminary and main results}\label{sec:setting}
In this section, we clarify the setting that we consider and the corresponding result that we mainly obtain. We start by introducing all of the notations.
\subsection{Notations}
 Let $d=1,2,3.$ Denote the torus $\mT^d:=\mR^d/\mZ^d$.  For $k\in\mZ^d$, denote the Fourier modes on $\mT^d$ by 
    \begin{align*}
   e_k(x):=\left\{\begin{array}{cc}
        {C_k} &  \text{if } \quad k=0,\\
       {C_k} e^{i\pi x\cdot k}  & \text{if }\quad |k|>0,
    \end{array}\right. 
     C_k:=\left\{\begin{array}{cc}
        \sqrt{2} &   \text{if } \quad k_1k_2=0,\\
      2  & \text{otherwise},
    \end{array}\right. 
        k\in\mZ^d, x\in \mT^d.
    \end{align*}
     Then $(e_k)_{k\in\mZ^d}$ is an orthonormal basis of $L^2(\mT^d,\mC)$. Meanwhile, observe
     \begin{align}
         \label{def:eigen}
         -\Delta e_k= 4 \pi^2 |k|^2 e_k=:\mu_k e_k. 
     \end{align}
     It says that $\mu_k:= 4 \pi^2 |k|^2$ are the eigenvalues of the Laplace operator $-\Delta$ on $\mT^d$.  Then we can derive that 
     \begin{align}
         \label{def:Bilaplace}
           -\Delta^2 e_k= - \mu_k^2e_k,\quad   (-\Delta)^\alpha e_k=\mu_k^{\alpha}e_k, \quad \alpha\in\mR, k\in \mZ^d\backslash\{0\}.
     \end{align}

    Given the probability space $(\Omega,\mathcal{F},\mP)$.  $\xi$ is {\it called space-time white noise} on $[0,T]\times\mT^d$ if for all $\phi,\psi\in L^2([0,T]\times\mT^d)$ hold that $\mE(\xi(\phi)\xi(\psi))=\langle \phi,\psi\rangle_{L^2([0,T]\times\mT^d)}$.    In other words,  let $(\beta_k)_{k \in \mZ^d\backslash\{0\}}$ be a sequence of i.i.d. real-valued Wiener processes. Let $W(t,x):=\sum_{k\in\mZ^d\backslash\{0\}}\beta_k(t)e_k(x),$ $x\in\mT^d$, $t \in [0,T]$.  Then its formal derivative $\xi_t:=\partial_tW$ is called {\it space-time white noise}.

 For $f\in L^1(\mT^d,\mC)$ denote the Fourier transform of $f$ as $\sF f(k):=\hat f(k):=\int_{\mT^2  } e^{-2\pi ik\cdot x}f(x)d x$, $k\in \mZ^d {\backslash\{0\}}$ and $\sF^{-1}$ corresponds to its inverse.     Define the semigroup 
    \begin{align}\label{def:semi-P}
        P_t u(x):=\sum_{k\in\mZ^d\backslash\{0\}}e^{-t\mu_k^2}u_ke_k(x),\quad\text{for}\ t\geq 0,\quad  u_k:=\langle u,e_k\rangle_{L^2(\mT^d)}.
    \end{align}
    One can equivalently write $P_tu=p_t\ast u$, where the kernel or Greensfunction is given by 
    \begin{align}
        \label{def:kernel}
        p_t(x):=\sF^{-1}(e^{-\frac{t}{2}|\pi\cdot|^4})(x)=\sum_{k\in\mZ^d\backslash\{0\}}e^{-t\mu_k^2}e_k(x).
    \end{align}
    For $N\in \mN$, let $\Pi_N$ be the orthogonal projection from $L^2(\mT^d,\mC)$ to its subspace $Span \left( (e_k), {|k|\leq N} \right)$. Define for $t\geq 0$
    \begin{align*}
       \Delta_N:=\Delta
       \Pi_N= \Pi_N\Delta
       ,
       \quad P_t^N:=P_t\Pi_N, \quad  P_t^N u(x):=\sum_{k\in\mZ^d\backslash\{0\},\atop |k|\leq N}e^{-t\mu_k^2}u_ke_k(x).
    \end{align*}  
Correspondingly, we have 
\begin{align}
    \label{def:ptN} P_t^N u=p_t^N\ast u,\qquad\qquad p_t^N (x):=\sum_{k\in\mZ^2\backslash\{0\} \atop |k|\leq N}e^{-t\mu_k^2}e_k(x).
\end{align}
Define the stochastic convolution as the following stochastic integral, either as a Fourier series expansion or in the sense of Walsh  
    \begin{align}
        \label{def:OU}
        U(t,x): & = \sum_{k\in\mZ^d\backslash\{0\}}\int_0^te^{- (t-s)\mu_k^2}d \beta_k(s)e_k(x)=\int_0^t\int_{\mT^d}p_{t-s}(x-y)\xi(ds,dy),\quad t\geq 0.
       \\ U^N(t,x):&=\sum_{k\in\mZ^d\backslash\{0\}, |k|\leq N}\int_0^te^{- (t-s) \mu_k^2}d \beta_k(s)e_k(x)=\int_0^t\int_{\mT^d}p^N_{t-s}(x-y)\xi(ds,dy).\label{def:OUN}
    \end{align}
    


Let $(\phi_j)_{j\geq-1}$ be the standard smooth dyadic partition of unity in $\mR^d$, i.e. 
a family of functions $\phi_j\in C^\infty(\mR^d)$ for $j \geq -1$, such that
$\phi_{-1}$ and $\phi_{0}$ are non-negative even functions such that the support of $\phi_{-1}$  is contained in
$B_{1/2}(0)$, the ball of radius $1/2$ around $0$, and the support of $\phi_0$ is contained in~$B_1(0)\backslash B_{1/4}(0)$; for $|j-i|>1$ we demand $\text{supp}\phi_j\cap\text{supp}\phi_i=\emptyset$; moreover, let $\phi_j(x)=\phi_0(2^{-j}x)$, for $x\in\mR^d$, $j\geq0$ and $\sum_{j\geq-1}^\infty \phi_j=1$ for every $x\in\mR^d$  (\cite{BCD}). Let $\cS'(\mT^d)$ denote the space of Schwartz distributions on $\mT^2$ and $\cS(\mT^d):= C^\infty(\mT^d)$. For $u\in\cS'(\mT^2)$, define 
    \begin{align}
        \label{def:block}\Delta_j:\cS'(\mT^d)\rightarrow C^\infty(\mT^d), \quad \Delta_ju:=\sF^{-1}(k\mapsto\phi_j(k)\sF(u)(k)).
    \end{align}
    Then the Besov space $\cB^\alpha_{p,q}$ on  $\mT^d$ for $p,q\in[1,\infty]$, $\alpha\in\mR$ is defined as 
    \begin{align}
        \label{def:Besov}\cB^\alpha_{p,q}:=\{f\in\cS'(\mT^d):\|f\|_{\cB^\alpha_{p,q}}:=\|(2^{j\alpha}\|\Delta_j f\|_{L^p(\mT^d)})_{j\geq -1}\|_{l^q}<\infty\}.
    \end{align}
   We let $C^\alpha:=\cB^\alpha_{\infty,\infty}$, for $\alpha\in\mR$. 
   We also use the notation
$C_x^\alpha \coloneqq C^\alpha $
to emphasize the spatial variable.
   Following from \cite[Chapter 2]{BCD},  we know that $C^\alpha$ is the known H\"older continuous function space, for $\alpha\in(0,1)$.  
   It is a well-known property of $\cB^\alpha_{p,q}$ that for any $g\in \cB^\alpha_{p,q}$,  we have 
   \begin{align}\label{prop:Besov}
       \|\partial^ng\|_{B_{p,q}^{\alpha-n}}\lesssim \|u\|_{B_{p,q}^\alpha}.
   \end{align}
   
   We always work on a finite time interval $[0,T]$. We write $a\lesssim b$ to indicate that there exists a positive constant $C$ such that $a \leq C b$.
For vectors $x,\, y\in \R^d$, we write $x\cdot y:=\sum_{i=1}^n x_i y_i$ and $|x|=\sqrt{x\cdot x}$.  We denote by $\lfloor z\rfloor$ the integer part of a real number $z$.
For any $N\in \mathbb{N}$ and $p\in [1,\infty]$, we denote by $L^p(\mR^d;\mR^N)$ (similarly on torus $L^p(\mT^d;\mR^N)$) the standard Lebesgue space; in the case where there is no risk of confusion on the dimension~$N$, we will simply write $L^p(\mR^d)$ (respectively $L^p(\mT^d)$) and denote by $\| \cdot\|_{L^p(\mR^d)}$  (respectively $\| \cdot\|_{L^p(\mT^d)}$) the corresponding norm. We denote by $C_T=C([0,T];\R^d)$ the space of continuous functions on $[0,T]$, endowed with the supremum norm $\| f\|_{C_T}=\sup_{t\in [0,T]} |f(t)|$. Given a Banach space $\cE$,
we denote by $C_T\cE_x$ the space of continuous functions $f:[0,T]\rightarrow\cE$ such that 
\begin{align*}
    \| f\|_{C_T\cE_x}:=\sup_{t\in [0,T]} \|f(t)\|_{\cE_x}<\infty.
\end{align*}
\subsection{Main results}
Here we introduce the solution and its corresponding approximation scheme that we consider.

\begin{definition}
    \label{def:mild-sol} 
    For $t\in[0,T]$, for  $k=0,1,\ldots, n$, $h=\frac{T}{n}$ and for $U$, $U^N$ from \eqref{def:OU}, \eqref{def:OUN}, let $v$, $u^N$ and~$u^{N,n}$ satisfy accordingly the following: 
 \begin{align}
        v(t)&=P_tu_0 + \int_0^tP_{t-s}\mathbf{G}( v(s)) d s + U(t) \qquad \qquad \qquad \qquad (\text{mild solution}) \label{def:sol-h}\\
        u^N(t)&=P_t^Nu_0+\int_0^tP_{t-s}^N\mathbf{G}( u^N(s)) d s +U^N(t) \qquad \qquad \qquad (\text{spectral Galerkin scheme})\label{def:sol-h-Galerkin}\\ 
      u^{N,n}(t_{k+1})&=P_h^Nu^{N,n}(t_{k})+P_{h}^N\mathbf{G}( u^{N,n}(t_k))  +U^N(t_{k+1})-P_h^{N}U(t_k)\quad (\text{exponential Euler scheme}). \label{def:sol-h-Ga-Eu}  
 \end{align}
 Observe that \eqref{def:sol-h-Ga-Eu} 
 is equivalent to
 \begin{align}
       u^{N,n}(t)&=P_t^Nu_0+\int_0^tP_{t-s}^N\mathbf{G}( u^{N,n}(k_n(s))) d s +U^N(t),\quad t\geq0, \label{eq:sol-n-N}
 \end{align}
  where we use $k_n(s):=\frac{\lfloor ns/T\rfloor}{n/T}$, $n\in\mN$ to round down to the nearest grid point.
\end{definition}

Our main result can be stated as follows:
\begin{theorem}
    \label{thm:main} Let $d=1,2,3$. Suppose $u_0\in C_x^{2-\frac{d}{2}}$ and $\|\bG\|_\infty, \|\partial^i \bG\|_\infty<\infty$, $i=1,\ldots,\lceil 3-\frac{d}{2}\rceil$. 
   For the solution $u^{N,n}$ to \eqref{def:sol-h-Ga-Eu}   
   and the solution $v$ to \eqref{def:sol-h}  we have for any sufficiently small $\varepsilon>0$
        \begin{align}
        \label{est:main-con-taming}
       \left( \mE\sup_{t\in[0,T]}\|u^{N,n}(t)-v(t)\|_{L^2(\mT^d)}^p\right)^\frac{1}{p}\leq C(N^{\frac{d}{2}-2+\varepsilon}+n^{-(\frac{6-d}{4}\wedge1)+\varepsilon})
    \end{align}
     where $C$ depends on $d,T, p, \varepsilon$. 
\end{theorem}

    \section{Tools and Auxiliary estimates}\label{sec:tools}
     Denote $[S, T]_\leq := \{(s, t) | S \leq s < t \leq T\}$ and $[S, T]_\leq^* := \{(s, t) | S \leq s < t \leq T, t-s\leq T-t\}$. For a function $A$ of one variable and $s \leq t$, we write $A_{s,t} \coloneqq A_t - A_s$ and for functions $A$ of
two variables and $s \leq u \leq t$, we denote $\delta A_{s,u,t} := A_{s,t} - A_{s,u} - A_{u,t}$. 
Furthermore, denote by $\mE_s$ the conditional expectation with respect to $\cF_s$.
    \begin{lemma}[Stochastic Sewing Lemma, \protect{\cite[Lemma 3.2]{DKP}}]\label{lem:SSL}
    Fix $p \geq 2$ and $0 \leq S < T \leq 1$. Let $ A : [S, T]_\leq \rightarrow L^p
(\Omega)$ be such that $A_{s,t}$ is $\cF_t$-measurable for all $(s, t) \in [S, T]_\leq$. Suppose that there exist $\varepsilon_1, \varepsilon_2 > 0$, $\delta_1, \delta_2 \geq 0$ and $C_1, C_2 < \infty$
satisfying $1/2 + \varepsilon_1 - \delta_1 > 0$, $1 + \varepsilon_2 -\delta_2 > 0$ and such that for all $(s, t) \in [S, T]_\leq^*$, $u \in [s, t]$ the following
bounds hold:
\begin{align}
    \label{con:SSL-1}\|A_{s,t}\|_{L^p(\Omega)}\leq C_1|T-t|^{-\delta_1}|t-s|^{\frac{1}{2}+\varepsilon_1},
    \\\label{con:SSL-2} \|\mE_s \delta A_{s,u,t}\|_{L^p(\Omega)}\leq C_2|T-t|^{ -\delta_2}|t-s|^{1+\varepsilon_2},
\end{align}
then there exists a unique $(\cF_t)_{t\in[S,T]}$-adapted process $\cA : [S, T] \rightarrow L^p
(\Omega)$ such that
$\cA_S = 0$ and that there exist $K_1, K_2$ such that for all $(s, t) \in [S, T]_\leq^*$ one has
\begin{align}
    \label{est:SSL-1}
    \|\cA_{s,t}-A_{s,t}\|_{L^p(\Omega)} & \leq K_1|T-t|^{ - \delta_1}|t-s|^{\frac{1}{2}+\varepsilon_1}+ K_2|T-t|^{ - \delta_2}|t-s|^{1+\varepsilon_2},
    \\ \label{est:SSL-2} \|\mE_s(\cA_{s,t}-A_{s,t})\|_{L^p(\Omega)} & \leq K_2|T-t|^{-\delta_2}|t-s|^{1+\varepsilon_2}.
\end{align}
Furthermore, there exists a constant $C $ depending only on $ p, \varepsilon_1, \varepsilon_2$, such that the above bounds
hold with $K_1 = CC_1, K_2 = CC_2$. Finally, there exists a constant $\tilde C$ depending only on
$p, \varepsilon_1, \varepsilon_2, \delta_1, \delta_2$, such that for all $(s, t) \in [S, T]_\leq $ one has
\begin{align}
    \label{est:SSL-3}\|\cA_{s,t}\|_{L^p(\Omega)} &
    \leq \tilde C
    \left(C_1|t-s|^{\frac{1}{2}+\varepsilon_1-\delta_1}+C_2|t-s|^{\frac{1}{2}+\varepsilon_2-\delta_2}
    \right).
\end{align}
    \end{lemma}

      We state a version of Gr\"onwall lemma that we use quite often in the later analysis.
      
    \begin{lemma}[Gr\"onwall Lemma. \protect{\cite[Proposition 3.4]{DKP}}]
    \label{lem:Gron}
        Let $V$ be a Banach space.  Let  $X, Y, Z\in  L_p(\Omega; C([0, T ]; V ))$, where $p\geq1$. Assume that there exists a Lipschitz continuous function $F$
on $V$ with Lipschitz constant $L_1$, a family $(\cS(s, t))_{0\leq s\leq t\leq T}$ of uniformly bounded linear operators
on $V$ with uniform bound $L_2$ and such that $(s, t) \mapsto \cS(s, t)v$ is measurable for any $v \in V$, a
measurable mapping $\tau  : [0, T ] \rightarrow [0, T ]$ such that $\tau (s) \leq s$ and that the following equality holds
for all $0 \leq t \leq T$:
\begin{align}
    \label{con:Gron}X_t-Y_t=Z_t+\int_0^t\cS(s,t)(F(X_{\tau(s)})-F(Y_{\tau(s)}))ds.
\end{align}
Then there exists a constant $C = C(p, L_1, L_2, T )$ such that
\begin{align}
    \label{est:Gron}
    \mE \sup_{t\in[0,T]}\|X_t-Y_t\|^p\leq C \mE \sup_{t\in[0,T]}\|Z_t\|^p.
\end{align}
    \end{lemma}

    \begin{lemma}
        \label{lem:auxi-semigroup}
        Let  $(P_t)_{t\geq0}$ be the semigroup defined in \eqref{def:semi-P} and $\mI_d$ denote the identity matrix in $\mR^d.$ Then for any $\alpha,\beta\in \mR$ so that $\alpha>\beta$ and for any $t>0$, $\theta\in(0,4)$ we have
        \begin{align}
            \label{est:semi-Hol}
            \|P_tf\|_{C^{\alpha}}\lesssim t^{-\frac{\alpha-\beta}{4}}\|f\|_{C^\beta},\quad \|(\mI_d-P_t)f\|_{C^{\beta}}\lesssim t^{\frac{\theta}{4}}\|f\|_{C^{\beta+\theta}}.
        \end{align}
    \end{lemma}
    \begin{proof}
     \cite[Lemma A.5, Lemma A.7]{GIP}  show the case where the semigroup is generated by the negative Laplacian $-\Delta$. For $-\Delta^2$ we obtain \eqref{est:semi-Hol} via the same proof except changing the rate from $t^{-\frac{\alpha-\beta}{2}}$ of  \cite[Lemma A.5, Lemma A.7]{GIP} to $t^{-\frac{\alpha-\beta}{4}}$  as in \eqref{est:semi-Hol}.  Therefore, the proof is identical to that of \cite[Lemma A.5, Lemma A.7]{GIP} and is omitted.
    \end{proof}
    For completeness, we state a standard regularity result for the Ornstein--Uhlenbeck process $U$. Since the precise values of the various exponents are crucial, we include the proof.
    \begin{lemma}
        \label{lem:auxi-semi-U}
        Define the processes $(U_t)_{t\geq0}$ and $(P_t)_{t\geq0}$, as in \eqref{def:OU} and \eqref{def:semi-P}. For $p\in[1,\infty)$, {$\lambda\in(0,2)$,}
        $\mu\in(0,d)$, and $\varepsilon\in(0,1)$ sufficiently small, there exist constants $C=C(\varepsilon,\lambda,T,p,d)$ and $\tilde C=C(\varepsilon,T,p,d,\mu)$ such that for any $0\leq s<t\leq T$, for $d=1,2,3$, the following holds:
        \begin{align}
            \label{est:Ut-Us}
            \mE \left\|U_t-U_s \right\|^p_{C_x^{2-\frac{d}{2}-\lambda-\varepsilon}}&\leq C|t-s|^{\frac{\lambda p}{4}},
            \\
           \mE\|U\|_{C_T^{\lambda/4}C_x^{2-\frac{d}{2}-\lambda-\varepsilon}}&\leq C, \label{est:U-space}
           \\
      \left( \mE \sup_{r\in[0,T]}\left\|(P_{t+s}-P_t)U_r \right\|^p_{C_x^{-(2-\frac{d}{2})+\varepsilon}}\right)^\frac{1}{p}&\leq \tilde C t^{-\frac{\mu}{4}-\frac{\varepsilon}{2}}s^{\frac{4-d+\mu}{4}}.
      \label{est:U-P}
        \end{align}
    \end{lemma}
    \begin{proof}
    We first derive \eqref{est:U-space} and \eqref{est:U-P}
    from \eqref{est:Ut-Us}. Choose
    \[
        0<\delta<\min\{2-\lambda,\varepsilon\},
    \]
    and let $q\geq p$ be sufficiently large that $q\delta/4>1$.
    Applying \eqref{est:Ut-Us} with $\lambda+\delta$,
    $\varepsilon-\delta$, and $q$ in place of $\lambda$,
    $\varepsilon$, and $p$, respectively, gives
    \[
        \mE
        \left\|U_t-U_s\right\|_
        {C_x^{2-\frac d2-\lambda-\varepsilon}}^q
        \lesssim |t-s|^{\frac{q(\lambda+\delta)}4}.
    \]
    Hence, by the quantitative version of Kolmogorov's
    continuity theorem and $U_0=0$, after passing to a
    modification, we obtain the stronger moment bound
    \[
        \left(
        \mE\|U\|_
        {C_T^{\lambda/4}
         C_x^{2-\frac d2-\lambda-\varepsilon}}^p
        \right)^{1/p}
        \lesssim 1.
    \]
    In particular, \eqref{est:U-space} follows.

    Next, since $0<4-d+\mu<4$, applying both bounds of
    \eqref{est:semi-Hol} yields
    \begin{align*}
        &\sup_{r\in[0,T]}
        \left\|(P_{t+s}-P_t)U_r\right\|_
        {C_x^{-(2-\frac d2)+\varepsilon}}
        \\
        &\qquad\lesssim
        t^{-\frac{\mu}{4}-\frac{\varepsilon}{2}}
        \sup_{r\in[0,T]}
        \left\|(P_s-\mI_d)U_r\right\|_
        {C_x^{-2+\frac d2-\mu-\varepsilon}}
        \\
        &\qquad\lesssim
        t^{-\frac{\mu}{4}-\frac{\varepsilon}{2}}
        s^{\frac{4-d+\mu}{4}}
        \sup_{r\in[0,T]}
        \left\|U_r\right\|_
        {C_x^{2-\frac d2-\varepsilon}}
        \\
        &\qquad\lesssim
        t^{-\frac{\mu}{4}-\frac{\varepsilon}{2}}
        s^{\frac{4-d+\mu}{4}}
        \|U\|_
        {C_T^{\varepsilon/8}
         C_x^{2-\frac d2-\varepsilon}}.
    \end{align*}
    Taking $L^p(\Omega)$-norms and applying the preceding
    moment bound with $\lambda=\varepsilon/2$ and with
    $\varepsilon$ replaced by $\varepsilon/2$ proves
    \eqref{est:U-P}.

    Thus, it remains to prove \eqref{est:Ut-Us}.
    It suffices to consider $p\geq2$ sufficiently large that
    $d/p<\varepsilon/2$, since the remaining cases follow
    from Jensen's inequality. Recalling that
    $C_x^\alpha=\cB_{\infty,\infty}^\alpha$ for
    $\alpha\in\mathbb R$, we first establish
    \begin{align}
        \label{est:Deltaj}
        \left(
        \mE|\Delta_j U_t(x)-\Delta_j U_s(x)|^p
        \right)^{1/p}
        \lesssim
        2^{-j(2-\frac d2-\lambda)}
        |t-s|^{\frac{\lambda}{4}},
    \end{align}
    uniformly in $x\in\mT^d$ and $j\geq-1$.

    Indeed, raising \eqref{est:Deltaj} to the $p$-th power,
    integrating in $x$, and summing over $j$ gives
    \begin{align*}
        &\mE
        \left\|U_t-U_s\right\|_
        {\cB_{p,p}^{2-\frac d2-\lambda-\frac{\varepsilon}{2}}}^p
        \\
        &\quad=
        \sum_{j\geq-1}
        2^{jp(2-\frac d2-\lambda-\frac{\varepsilon}{2})}
        \int_{\mT^d}
        \mE|\Delta_j U_t(x)-\Delta_j U_s(x)|^p\,dx
        \\
        &\quad\lesssim
        |t-s|^{\frac{p\lambda}{4}}
        \sum_{j\geq-1}2^{-jp\varepsilon/2}
        \lesssim |t-s|^{\frac{p\lambda}{4}}.
    \end{align*}
    Since $d/p<\varepsilon/2$, the Besov embeddings
    \[
        \cB_{p,p}^{2-\frac d2-\lambda-\frac{\varepsilon}{2}}
        \hookrightarrow
        C_x^{2-\frac d2-\lambda-\frac{\varepsilon}{2}-\frac dp}
        \hookrightarrow
        C_x^{2-\frac d2-\lambda-\varepsilon}
    \]
    imply \eqref{est:Ut-Us}. We therefore only need to
    prove \eqref{est:Deltaj}.

    Applying Gaussian hypercontractivity and It\^o's
    isometry, we obtain
    \begin{align*}
        \left(
        \mE|\Delta_j U_t(x)-\Delta_j U_s(x)|^p
        \right)^{2/p}
        &\lesssim
        \mE|\Delta_j U_t(x)-\Delta_j U_s(x)|^2
        \\
        &\lesssim
        \mE\left|
        \int_s^t\int_{\mT^d}
        \Delta_j p_{t-r}(x-y)\,\xi(dr,dy)
        \right|^2
        \\
        &\qquad\quad+
        \mE\left|
        \int_0^s\int_{\mT^d}
        \Delta_j(p_{t-r}-p_{s-r})(x-y)\,\xi(dr,dy)
        \right|^2
        \\
        &=
        \int_s^t
        \|\Delta_j p_{t-r}\|_{L^2(\mT^d)}^2\,dr
        +
        \int_0^s
        \|\Delta_j(p_{t-r}-p_{s-r})\|_{L^2(\mT^d)}^2\,dr.
    \end{align*}
    By Parseval's identity and \eqref{def:kernel},
    \begin{align*}
        \left(
        \mE|\Delta_j U_t(x)-\Delta_j U_s(x)|^p
        \right)^{2/p}
        &\lesssim
        \int_s^t
        \sum_{k\in\mZ^d\setminus\{0\}}
        \phi_j^2(k)e^{-2(t-r)\mu_k^2}\,dr
        \\
        &\qquad\quad+
        \int_0^s
        \sum_{k\in\mZ^d\setminus\{0\}}
        \phi_j^2(k)e^{-2(s-r)\mu_k^2}
        (1-e^{-(t-s)\mu_k^2})^2\,dr
        \\
        &\quad\lesssim
        \sum_{k\in\mZ^d\setminus\{0\}}
        \phi_j^2(k)\min\{|t-s|,\mu_k^{-2}\}
        \\
        &\qquad\quad+
        \sum_{k\in\mZ^d\setminus\{0\}}
        \phi_j^2(k)
        \min\{|t-s|^2\mu_k^2,\mu_k^{-2}\}
        \\
        &\lesssim
        \sum_{k\in\mZ^d\setminus\{0\}}
        \phi_j^2(k)\min\{|t-s|,\mu_k^{-2}\}.
    \end{align*}
    Here we used $1-e^{-x}\leq\min\{x,1\}$ for $x\geq0$,
    together with
    \[
        \min\{h^2\mu_k^2,\mu_k^{-2}\}
        \leq \min\{h,\mu_k^{-2}\},
        \qquad h\geq0.
    \]

    By \eqref{def:eigen},
    $\mu_k=4\pi^2|k|^2$ for
    $k\in\mZ^d\setminus\{0\}$.
    Hence, using the support of $\phi_j$, for any
    $\theta\in[0,1]$ we have
    \begin{align*}
        \left(
        \mE|\Delta_j U_t(x)-\Delta_j U_s(x)|^p
        \right)^{2/p}
        &\lesssim
        2^{dj}\min\{2^{-4j},|t-s|\}
        \\
        &\leq
        2^{dj}2^{-4j(1-\theta)}|t-s|^\theta
        \\
        &=
        |t-s|^\theta 2^{(4\theta-4+d)j}.
    \end{align*}
    Taking $\theta=\lambda/2\in(0,1)$ and square roots
    yields
    \begin{align}
        \label{est:Delta-UN}
        \left(
        \mE|\Delta_j U_t(x)-\Delta_j U_s(x)|^p
        \right)^{1/p}
        \lesssim
        |t-s|^{\frac{\lambda}{4}}
        2^{(\lambda-2+\frac d2)j}.
    \end{align}
    This proves \eqref{est:Deltaj} and completes the proof.
\end{proof}
    \section{Proof of the main results}\label{sec:main-proof}
      Let us recall that we need $d\in\{1, 2, 3\}$, as we used $4-d>0$ in various estimates. We will split the main error term in various parts, and need  different estimates in the proof of the main results, we  first state the  auxiliary lemmas in the following and finish the proof of \Cref{thm:main}  in the end. 
   
      First, define the discretization schemes using the full OU-process by
    \begin{align}
        \hat u^N(t)&=P_t^Nu_0+\int_0^tP_{t-s}^N\mathbf{G}( \hat u^N(s)) d s +U(t),
        \label{def:sol-h-G}\\
        \hat u^{N,n}(t)&=P_t^Nu_0+\int_0^tP_{t-s}^N\mathbf{G}( \hat u^{N,n}(k_n(s))) d s +U(t)
        ,\quad t\geq0, \label{eq:sol-n-N-hat}  
 \end{align}
    The following result shows that all of the terms that are introduced above are well-defined. 
\begin{lemma}
    \label{lem:SPDE} 
    Suppose $u_0\in   C_x^{2-\frac{d}{2}}$ and $\|\bG\|_\infty, \|\partial \bG\|_\infty<\infty$,
  then there exists a unique mild solution  $u^{N}$ to \eqref{def:sol-h-Galerkin}, 
    $\hat u^{N}$ to \eqref{def:sol-h-G} accordingly; moreover, for small $\varepsilon\in(0,\frac{1}{2})$, $\lambda\in(0, 2)$ and $p\geq1$, there exists a constant $C(d, p,T,\|\bG\|_\infty, \|\partial \bG\|_\infty)$  so that the following holds:
   \begin{align*}
       \sup_{N\in \mN}\mE\|u^N\|^p_{C_T^{\lambda/4} C_x^{2-\frac{d}{2}-\lambda-\varepsilon}}+   \sup_{N\in \mN}\mE\|\hat u^N\|^p_{C_T^{\lambda/4} C_x^{2-\frac{d}{2}-\lambda-\varepsilon}}\leq C(1+\mE\|u_0\|_{C_x^1}^p).
   \end{align*}
\end{lemma} 
\begin{proof}
   We follow the standard idea (see also  \cite[Proposition 4.1]{DKP}).  Here for the existence and uniqueness we apply fixed point argument.  By the global Lipschitz bound on $\bG$ and the fact that $\sup_N\|P_t^Nu\|_\infty\leq \|u\|_\infty$ one finds a unique fixed point of the mild formulation of    $u^{N}$ to \eqref{def:sol-h-Galerkin} 
   and $\hat u^{N}$ to \eqref{def:sol-h-G} in $C_TL_x^\infty$ for $T$ chosen small enough. For
any arbitrary time horizon $T > 0$, we can glue the solutions on subintervals together yielding a solution globally.  Then plugging the solution back in the mild formulation and
using the semigroup estimates for $P^N$ instead of $P$, one obtains the claimed regularity bounds by a similar calculation \Cref{lem:auxi-semi-U}.
\end{proof}

We split our error term in various terms 
    \begin{align}
        \label{est:uNn-v} u^{N,n}-v=u^{N,n}-  \hat u^{N,n}+   \hat u^{N,n}- \hat u^N+ \hat u^N-u^N+u^N-v=:I_1+I_2+I_3+I_4.
    \end{align}
We now bound all terms separately, where $I_2$ will be the critical term. The term $I_4$ is a fairly standard estimate based on a Galerkin projection. The two terms $I_1$ and  $I_3$ are also only error estimates using a Galerkin projection on the OU-process only. One is for the discrete scheme and the other for the continuous one.

In contrast to that in $I_2$ we  use a Girsanov transformation and the Stochastic Sewing Lemma to obtain an optimal rate of convergence. This term is the cricial part of our analysis.

    \subsection{Estimate of $I_1$}
    \begin{lemma} \label{lem:I1-est}
        Suppose $u_0 \in  C_x^{2-\frac{d}{2}}$ and $\|\bG\|_\infty, \|\partial\bG\|_\infty<\infty$. For $u^{N,n}$ from \eqref{eq:sol-n-N} and $\hat u^{N,n}$ from \eqref{eq:sol-n-N-hat},  we have for sufficiently small $\varepsilon>0$, for any $p\geq1$
        \begin{align}\label{est:I1}
             \left( \mE\sup_{t\in[0,T]}\|u^{N,n}(t)-  \hat u^{N,n}(t)\|_{L^2(\mT^d)}^p\right)^\frac{1}{p} 
             \lesssim 
             N^{-2+\frac{d}{2}+\varepsilon}.
        \end{align}
    \end{lemma}

    \begin{proof}
        By \eqref{eq:sol-n-N-hat}, \eqref{eq:sol-n-N} we obtain 
             \begin{align*}
              u^{N,n}(t)-  \hat u^{N,n}(t) =
              \left[  U(t) - U^N(t) \right] +
             \int_0^t 
             P_{t-s}^N
             \left[ \mathbf{G}(u^{N,n}(k_n(s))) 
             - \mathbf{G}( \hat u^{N,n}(k_n(s))) 
             \right]
             ds.
        \end{align*}
        Thus, by applying \Cref{lem:Gron} together with Lipschitz continuity of $\bG$, we derive  
    \begin{align*}
         \mE\sup_{t\in[0,T]}\|u^{N,n}(t)-  \hat u^{N,n}(t)\|_{L^2(\mT^d)}^p
             & \lesssim
             \mE \sup_{t\in[0,T]}
             \left \| U(t) - U^N(t)
             \right \|_{L^2(\mT^d)}^p.
    \end{align*}
    By the definition of $U$ and $U^N$, we have 
    \begin{align*}
        (U_t-U_t^N)(x)=\sum_{k\in\mZ^d\backslash\{0\}|k|>N}\int_0^te^{-(t-r)\mu_k^2}d\beta_k(r)e_k(x),
    \end{align*}
    it implies for any $0\leq s\leq t$ (similar to the analysis of proving Lemma \ref{lem:auxi-semi-U}), by Gaussian hypercontractivity, It\^o's isometry and Parseval identity,
    \begin{align}\label{est:U-U^N-1-41}
       & \big(\mE|\Delta_j(U_t-U_t^N)-(U_s-U_s^N)|^p(x)\big)^\frac{1}{p}
      \nonumber  \\&\lesssim \Big(\sum_{k\in\mZ^d\backslash\{0\},|k|>N}\int_s^t\phi_j^2(k)e^{-2\mu_k^2(t-r)}dr+\sum_{k\in\mZ^d\backslash\{0\},|k|>N}\int_0^s\phi_j^2(k) e^{-\mu_k^2(s-r)}(1-e^{-\mu_k^2(t-r)})dr\Big)^\frac{1}{2} \nonumber\\&\lesssim 
        \Big(\sum_{|k|>N}\frac{1-e^{-2\mu_k^2(t-s)}}{2\mu_k^2}+\sum_{k\in\mZ^2\backslash\{0\},|k|>N}\frac{(1-e^{-2\mu_k^2(t-s)})(1-e^{-\mu_k(t-s)})^2}{2\mu_k^2}\Big)^\frac{1}{2}
      \nonumber  \\&\lesssim \Big((t-s)^{\epsilon_0}\sum_{|k|>N}\frac{1}{K^{4-2\epsilon_0}}\Big)^\frac{1}{2}
      \nonumber  \\&\lesssim N^{\frac{d}{2}-2+\epsilon_0}(t-s)^{\frac{\epsilon_0}{2}}
    \end{align}
    using that for any $\epsilon_0\in[0,1]$, $1-e^{-x}\leq x^{\epsilon_0}$, $x\geq0$. Moreover, we also have the trivial bound on $N$
    \begin{align}\label{est:U-U^n-2-41}
    &\big(\mE|\Delta_j(U_t-U_t^N)-(U_s-U_s^N)|^p(x)\big)^\frac{1}{p}
    \nonumber\\&\lesssim   \Big(\sum_{|k|>N}\int_s^t\phi_j^2(k)e^{-2\mu_k^2(t-r)}dr+\sum_{|k|>N}\int_0^s\phi_j^2(k) e^{-\mu_k^2(s-r)}(1-e^{-\mu_k^2(t-r)})dr\Big)^\frac{1}{2}
    \nonumber\\&\lesssim \Big(\int_s^t2^{2j}e^{-32\pi^42^{4j}(t-r)}dr+\int_0^s2^{2j}e^{-32\pi^42^{4j}(s-r)}(1-e^{-32\pi^42^{4j}(t-r)})dr\Big)^{\frac{1}{2}}
 \nonumber   \\&\lesssim 2^{-j}.
    \end{align}
    where in the penultimate inequality we applied the change of variables. Therefore, 
    using interpolation between \eqref{est:U-U^N-1-41} and \eqref{est:U-U^n-2-41}, we derive  for any $\epsilon'\in(0,1)$
    \begin{align*}
       &\big(\mE|\Delta_j(U_t-U_t^N)-(U_s-U_s^N)|^p(x)\big)^\frac{1}{p}\lesssim 2^{-j\epsilon'}N^{(\frac{d}{2}-2+\epsilon)(1-\epsilon')}(t-s)^{(1-\epsilon')\frac{\epsilon_0}{2}} . 
    \end{align*}
    By choosing $\epsilon,\epsilon'$ sufficiently small and $p$ sufficiently large, we can repeat the argument of proving the estimate \eqref{est:Delta-UN} to obtain 
    \begin{align}
        \label{est:U-U^N}
        \Big(\mE \|U-U^N\|_{C_T^{\epsilon''}C_x^{\epsilon''}}^p\Big)^\frac{1}{p}\lesssim N^{\frac{d}{2}-2+\epsilon}
    \end{align}
    by taking small $\epsilon''>0$.
    \end{proof}
    \subsection{Estimate of $I_2$}
\begin{lemma}\label{lem:I2-est}
    Suppose $u_0\in C_x^{2-\frac d2}$ and
    $\|\bG\|_\infty,\|\partial^i\bG\|_\infty<\infty$,
    $i=1,\ldots,\lceil3-\frac d2\rceil$.
    For $\hat u^{N,n}$ from \eqref{eq:sol-n-N-hat} and
    $\hat u^N$ from \eqref{def:sol-h-G}, for
    $\varepsilon\in(0,\frac12)$ and $p\geq1$, we have
    \begin{align}\label{est:I2}
        \left(\mE\sup_{t\in[0,T]}
        \|\hat u^N(t)-\hat u^{N,n}(t)\|_{L^2(\mT^d)}^p
        \right)^{1/p}
        \leq Cn^{-\left(\frac{6-d}{4}\wedge1\right)+\varepsilon},
    \end{align}
    where $C=C(T,p,d,\varepsilon,u_0,\bG)$ is independent
    of $N,n$, and its dependence on $\bG$ is through the
    indicated derivative bounds.
\end{lemma}

\begin{proof}[Proof of \Cref{lem:I2-est}]
    We work with sufficiently small $\varepsilon>0$ and
    prove the estimate with $3\varepsilon$ in place of
    $\varepsilon$. The assertion then follows by decreasing
    and relabelling $\varepsilon$.
    It is enough to consider arbitrarily large $p$, since
    the remaining cases follow from Jensen's inequality.
    We have
    \begin{align}
        \hat u^N(t)-\hat u^{N,n}(t)
        &=\int_0^tP_{t-s}^N
        \bigl(\bG(\hat u^N(s))
        -\bG(\hat u^N(k_n(s)))\bigr)\,ds
        \nonumber\\
        &\quad+\int_0^tP_{t-s}^N
        \bigl(\bG(\hat u^N(k_n(s)))
        -\bG(\hat u^{N,n}(k_n(s)))\bigr)\,ds
        \nonumber\\
        &=:I_{21}(t)+I_{22}(t).
        \label{def:I2-1-2}
    \end{align}
    Our idea is to bound $\hat u^N-\hat u^{N,n}$ via
    Gr\"onwall \Cref{lem:Gron}, in which $I_{22}$ can be
    absorbed, while $I_{21}$ takes the role of the error
    term $Z$ and needs to be estimated.
    We divide the proof into the following steps.

    \begin{itemize}
        \item[\textbf{Step 1.}]
        \textbf{Reduction via Girsanov Theorem:}
    \end{itemize}
    We follow the idea of \cite[Corollary 4.7]{DKP}.
    Here and below $\Pi_N$ acts on the retained nonzero
    Fourier modes. Define
    \begin{align*}
        d\mQ:=\rho\,d\mP
        :=\exp\Bigg(&-\int_0^T\int_{\mT^d}
        [\Pi_N\bG(\hat u^N(s))](y)\,\xi(ds,dy)
        \\
        &-\frac12\int_0^T
        \|\Pi_N\bG(\hat u^N(s))\|_{L^2}^2\,ds
        \Bigg)\,d\mP.
    \end{align*}
    Novikov's condition follows from
    \[
        \int_0^T\|\Pi_N\bG(\hat u^N(s))\|_{L^2}^2\,ds
        \leq T|\mT^d|\|\bG\|_\infty^2.
    \]
    By Girsanov's theorem
    (see \cite[Theorem 10.14]{DPG}), the shifted noise
    \[
        \xi(ds,dy)+[\Pi_N\bG(\hat u^N(s))](y)\,ds\,dy
    \]
    is white noise on the same mean-zero noise space under
    $\mQ$, and the inverse moments of $\rho$ under $\mQ$
    are bounded uniformly in $N$. Consequently,
    \[
        \operatorname{law}^{\mQ}(\hat u^N)
        =\operatorname{law}^{\mP}(U_t+P_t^Nu_0).
    \]
    Define the functional on a process $\cZ$ by
    \[
        \cH(\cZ):=\sup_{t\in[0,T]}
        \left\|\int_0^tP_{t-s}^N
        \bigl(\bG(\cZ(s))-\bG(\cZ(k_n(s)))\bigr)\,ds
        \right\|_{L^2}.
    \]
    By Girsanov's theorem and Cauchy--Schwarz,
    \begin{align*}
        \mE|\cH(\hat u^N)|^p
        &=\mE^{\mQ}\bigl[\rho^{-1}|\cH(\hat u^N)|^p\bigr]
        \\
        &\leq(\mE^{\mQ}\rho^{-2})^{1/2}
        (\mE^{\mQ}|\cH(\hat u^N)|^{2p})^{1/2}
        \\
        &\lesssim
        (\mE^{\mP}|\cH(U_t+P_t^Nu_0)|^{2p})^{1/2}.
    \end{align*}

    Therefore, if we denote, for $0\leq t\leq t^*\leq T$,
    \begin{align}\label{def:I21-O}
        \hat I_{21}(t,t^*)
        :=\int_0^tP_{t^*-r}^N
        \bigl[\bG(U_r+P_r^Nu_0)
        -\bG(U_{k_n(r)}+P_{k_n(r)}^Nu_0)\bigr]dr,
    \end{align}
    then the estimate of $I_{21}$ is reduced to
    \begin{align}\label{est:I-21-r}
        \left(\mE\sup_{t\leq T}\|I_{21}(t)\|_{L^2}^p
        \right)^{1/p}
        \lesssim
        \left(\mE\sup_{t\leq T}
        \|\hat I_{21}(t,t)\|_{L^2}^{2p}\right)^{1/(2p)}.
    \end{align}

    We separate the change in the deterministic initial-data
    term. By $L^2$-contractivity of $P^N$ and boundedness
    of $\bG'$, pathwise,
    \begin{align*}
        &\sup_{t\leq T}\left\|\int_0^tP_{t-r}^N
        \bigl[\bG(U_{k_n(r)}+P_r^Nu_0)
        -\bG(U_{k_n(r)}+P_{k_n(r)}^Nu_0)\bigr]dr
        \right\|_{L^2}
        \\
        &\qquad\lesssim
        \int_0^T\|(P_r^N-P_{k_n(r)}^N)u_0\|_{L^2}\,dr
        \\
        &\qquad\leq\frac{T}{n}
        \int_0^T\|\partial_rP_r^Nu_0\|_{L^2}\,dr
        \lesssim n^{-1}.
    \end{align*}
    Indeed, uniformly in $N$,
    \[
        \|\partial_rP_r^Nu_0\|_{L^2}
        \lesssim r^{-1+\frac{2-d/2-\varepsilon}{4}}
        \|u_0\|_{\cB_{2,2}^{2-d/2-\varepsilon}},
    \]
    whose right-hand side is integrable near zero.

    \begin{itemize}
        \item[\textbf{Step 2.}]
        \textbf{Estimates of $\hat I_{21}$ via
        \emph{Stochastic Sewing}:}
    \end{itemize}
    As before, denote $\tilde U_t:=U_t+P_t^Nu_0$.
    In view of the deterministic estimate at the end of
    Step 1, we first freeze the initial-data term at the
    integration time and show that, for $t^*\in(0,T]$
    and $0\leq s<t\leq t^*$,
    \begin{align}\label{est:step2-final}
        &\left(\mE\left\|\int_s^tP_{t^*-r}^N
        \bigl[\bG(\tilde U_r)
        -\bG(U_{k_n(r)}+P_r^Nu_0)\bigr]\,dr
        \right\|_{L^\infty}^p\right)^{1/p}
        \nonumber\\
        &\qquad\lesssim
        (t-s)^{5\varepsilon/4}
        n^{-\left(\frac{6-d}{4}\wedge1\right)+3\varepsilon}.
    \end{align}
    By the Besov embedding, it is enough to prove, for
    arbitrarily large $p$, uniformly in $x$ and $j\geq-1$,
    \begin{align}\label{est:step2-inter}
        &\left(\mE\left|\int_s^t\Delta_jP_{t^*-r}^N
        \bigl[\bG(\tilde U_r)
        -\bG(U_{k_n(r)}+P_r^Nu_0)\bigr](x)\,dr
        \right|^p\right)^{1/p}
        \nonumber\\
        &\qquad\lesssim 2^{-j\varepsilon}
        (t-s)^{5\varepsilon/4}
        n^{-\left(\frac{6-d}{4}\wedge1\right)+3\varepsilon}.
    \end{align}
    Indeed, integrating the $p$-th power in $x$ and summing
    over $j$ gives the same bound in
    $L^p(\Omega;\cB_{p,p}^{\varepsilon/2})$.
    For sufficiently large $p$,
    \[
        \cB_{p,p}^{\varepsilon/2}
        \hookrightarrow C_x^{\varepsilon/2-d/p}
        \hookrightarrow L^\infty.
    \]
    To prove \eqref{est:step2-inter}, define
    \begin{align}\label{def:A}
        A_{s,t}^{(j)}(x):=\mE_s\int_s^t
        \Delta_jP_{t^*-r}^N
        \bigl[\bG(\tilde U_r)
        -\bG(U_{k_n(r)}+P_r^Nu_0)\bigr](x)\,dr.
    \end{align}

    We first record the estimates needed to verify
    \eqref{con:SSL-1}. The corrected
    \Cref{lem:auxi-semi-U}, with
    $\lambda=\min\{2,4-d\}-4\varepsilon\in(0,2)$,
    gives
    \begin{align}\label{est:tiled-Ut-Us}
        \left(\mE\|U_t-U_s\|_
        {C_x^{-2+d/2+3\varepsilon}}^p\right)^{1/p}
        \lesssim
        |t-s|^{\left(\frac{6-d}{4}\wedge1\right)-1/2-\varepsilon}.
    \end{align}
    In dimension one we also use the embedding into the
    indicated, weaker spatial space.
    Moreover, since $\Pi_N$ is a contraction on
    $\cB_{2,2}^\alpha$, the corrected
    \eqref{est:U-space} and the Besov embedding imply
    \begin{align}\label{est:tiled-U-space}
        \left(\mE\sup_{r,v\in[0,T]}
        \|U_r+P_v^Nu_0\|_
        {\cB_{2,2}^{2-d/2-\varepsilon}}^p\right)^{1/p}
        \lesssim1.
    \end{align}
    We use here the stronger, finite-moment version of
    \eqref{est:U-space} proved in
    \Cref{lem:auxi-semi-U}.
    Finally, \eqref{est:U-P}, with its loss parameter
    replaced by $3\varepsilon$, gives, for $\mu\in(0,d)$,
    \begin{align}\label{est:tiledU-P}
        \left(\mE\sup_{r\leq T}
        \|(P_{t+s}-P_t)U_r\|_
        {C_x^{-2+d/2+3\varepsilon}}^p\right)^{1/p}
        \lesssim
        t^{-\mu/4-3\varepsilon/2}
        s^{(4-d+\mu)/4}.
    \end{align}

    For the nonlinear estimate, the mean value formula and
    Bony's decomposition \cite[Section 2]{BCD} give
    \begin{align}\label{est:product-est}
        \|fg\|_{\cB_{2,2}^{-2+d/2+2\varepsilon}}
        \lesssim
        \|f\|_{\cB_{2,2}^{2-d/2-\varepsilon}}
        \|g\|_{C_x^{-2+d/2+3\varepsilon}}.
    \end{align}
    In particular, whenever the right-hand side is finite,
    \begin{align}\label{G:Lip}
        &\|\bG(u_1)-\bG(u_2)\|_
        {\cB_{2,2}^{-2+d/2+2\varepsilon}}
        \nonumber\\
        &\quad\lesssim
        \left(1+\|u_1\|_{\cB_{2,2}^{2-d/2-\varepsilon}}^2
        +\|u_2\|_{\cB_{2,2}^{2-d/2-\varepsilon}}^2\right)
        \|u_1-u_2\|_{C_x^{-2+d/2+3\varepsilon}}.
    \end{align}
    Indeed, write
    \[
        \bG(u_1)-\bG(u_2)
        =\int_0^1\bG'(\theta u_1+(1-\theta)u_2)
        (u_1-u_2)\,d\theta
    \]
    and use
    \[
        \|\bG'(u)\|_{\cB_{2,2}^{2-d/2-\varepsilon}}
        \lesssim1+\|u\|_{\cB_{2,2}^{2-d/2-\varepsilon}}^2.
    \]
    For $d=2,3$, this follows from the fractional Sobolev
    composition estimate of order less than one. For $d=1$,
    it follows from the chain rule, the one-dimensional
    Sobolev embedding, and the assumed bound on $\bG'''$.

    We also need the following uniform smoothing bound:
    \[
        \|P_\ell^N f\|_{C_x^\varepsilon}
        \lesssim
        \ell^{-(2-\varepsilon)/4}
        \|f\|_{\cB_{2,2}^{-2+d/2+2\varepsilon}},
        \qquad0<\ell\leq T.
    \]
    It follows from Bernstein's inequality and
    $L^2$-contractivity of $\Pi_N$, since
    \[
        \|\Delta_jP_\ell^Nf\|_\infty
        \lesssim2^{jd/2}e^{-c\ell2^{4j}}
        \|\Delta_jf\|_{L^2}.
    \]
    Thus no uniform H\"older bound for the sharp
    projection $\Pi_N$ is required.

    We now consider the same two cases as before.
    \begin{enumerate}
        \item $t-s\leq3T/n$.

        By \eqref{G:Lip}, \eqref{est:tiled-U-space},
        \eqref{est:tiled-Ut-Us}, and H\"older's inequality,
        \begin{align}\label{est:Ast-Lp1}
            \|A_{s,t}^{(j)}(x)\|_{L^p(\Omega)}
            &\lesssim2^{-j\varepsilon}
            (t^*-t)^{-(2-\varepsilon)/4}
            (t-s)
            n^{-\left(\frac{6-d}{4}\wedge1\right)+1/2+\varepsilon}
            \nonumber\\
            &\lesssim2^{-j\varepsilon}
            (t^*-t)^{-(2-\varepsilon)/4}
            (t-s)^{1/2+\varepsilon}
            n^{-\left(\frac{6-d}{4}\wedge1\right)+2\varepsilon}
            \nonumber\\
            &\lesssim2^{-j\varepsilon}
            (t^*-t)^{-(2-\varepsilon)/4}
            (t-s)^{1/2+\varepsilon}
            n^{-\left(\frac{6-d}{4}\wedge1\right)+3\varepsilon}.
        \end{align}
        Here the second inequality uses $t-s\leq3T/n$.

        \item $t-s>3T/n$.

        Denote $\tilde k_n(s):=k_n(s)+2T/n$.
        For $r\in[\tilde k_n(s),t]$ we have
        \[
            0\leq r-k_n(r)\leq T/n\leq k_n(r)-s,
            \qquad k_n(r)-s\geq(r-s)/2.
        \]
        Split
        \begin{align}\label{est:Ast-2-12}
            A_{s,t}^{(j)}(x)
            &=\mE_s\left(\int_s^{\tilde k_n(s)}
            +\int_{\tilde k_n(s)}^t\right)
            \Delta_jP_{t^*-r}^N
            \bigl[\bG(\tilde U_r)
            -\bG(U_{k_n(r)}+P_r^Nu_0)\bigr](x)\,dr
            \nonumber\\
            &=:S_1+S_2.
        \end{align}
        Since $\tilde k_n(s)-s\leq2T/n$, the first case gives
        \begin{align}\label{est:Ast-Lp2}
            \|S_1\|_{L^p(\Omega)}
            \lesssim2^{-j\varepsilon}
            (t^*-t)^{-(2-\varepsilon)/4}
            (t-s)^{1/2+\varepsilon}
            n^{-\left(\frac{6-d}{4}\wedge1\right)+3\varepsilon}.
        \end{align}
        For $S_2$, use the Markov property of $U$. Set
        \begin{align}\label{def:Q}
            Q(v):=\int_0^v\|p_\ell\|_{L^2(\mT^d)}^2\,d\ell,
        \end{align}
        and let $\cP_{Q(v)}$ denote Gaussian convolution
        on $\mathbb R$ with variance $Q(v)$.
        Conditional Gaussianity gives
        \begin{align}\label{def:S21S22}
            S_2
            &=\int_{\tilde k_n(s)}^t\Delta_jP_{t^*-r}^N
            \Big[(\cP_{Q(r-s)}\bG)
            (P_{r-s}U_s+P_r^Nu_0)
            \nonumber\\[-1mm]
            &\hspace{39mm}
            -(\cP_{Q(r-s)}\bG)
            (P_{k_n(r)-s}U_s+P_r^Nu_0)\Big](x)\,dr
            \nonumber\\
            &\quad+\int_{\tilde k_n(s)}^t\Delta_jP_{t^*-r}^N
            \Big[((\cP_{Q(r-s)}-\cP_{Q(k_n(r)-s)})\bG)
            (P_{k_n(r)-s}U_s+P_r^Nu_0)\Big](x)\,dr
            \nonumber\\
            &=:S_{21}+S_{22}.
        \end{align}
        Since Gaussian convolution preserves the bounds
        on the derivatives of $\bG$, the same argument as
        for \eqref{G:Lip} gives, uniformly in $v\geq0$,
        \begin{align}\label{PtG:Lip}
            &\|(\cP_v\bG)(u_1)-(\cP_v\bG)(u_2)\|_
            {\cB_{2,2}^{-2+d/2+2\varepsilon}}
            \nonumber\\
            &\quad\lesssim
            \left(1+\|u_1\|_{\cB_{2,2}^{2-d/2-\varepsilon}}^2
            +\|u_2\|_{\cB_{2,2}^{2-d/2-\varepsilon}}^2\right)
            \|u_1-u_2\|_{C_x^{-2+d/2+3\varepsilon}}.
        \end{align}
        The Sobolev moments of both arguments in $S_{21}$
        are bounded uniformly, by \eqref{est:tiled-U-space}
        and contractivity of $P$ on $\cB_{2,2}^\alpha$.
        Therefore,
        \begin{align}\label{est:S21}
            \|S_{21}\|_{L^p(\Omega)}
            &\lesssim2^{-j\varepsilon}
            \int_{\tilde k_n(s)}^t
            (t^*-r)^{-(2-\varepsilon)/4}
            \nonumber\\
            &\qquad\times
            \|(P_{r-s}-P_{k_n(r)-s})U_s\|_
            {L^{2p}(\Omega;C_x^{-2+d/2+3\varepsilon})}\,dr.
        \end{align}
        Taking $\mu=\min\{d,2\}-12\varepsilon\in(0,d)$
        in \eqref{est:tiledU-P}, we obtain
        \begin{align*}
            &\|(P_{r-s}-P_{k_n(r)-s})U_s\|_
            {L^{2p}(\Omega;C_x^{-2+d/2+3\varepsilon})}
            \\
            &\quad\lesssim
            (r-k_n(r))^{\left(\frac{6-d}{4}\wedge1\right)-3\varepsilon}
            (k_n(r)-s)^{-\min\{d,2\}/4+3\varepsilon/2}
            \\
            &\quad\lesssim
            n^{-\left(\frac{6-d}{4}\wedge1\right)+3\varepsilon}
            (k_n(r)-s)^{-1/2+\varepsilon}.
        \end{align*}
        Consequently,
        \begin{align}\label{est:S21-fur}
            \|S_{21}\|_{L^p(\Omega)}
            &\lesssim2^{-j\varepsilon}
            (t^*-t)^{-(2-\varepsilon)/4}
            n^{-\left(\frac{6-d}{4}\wedge1\right)+3\varepsilon}
            \int_{\tilde k_n(s)}^t(r-s)^{-1/2+\varepsilon}\,dr
            \nonumber\\
            &\lesssim2^{-j\varepsilon}
            (t^*-t)^{-(2-\varepsilon)/4}
            (t-s)^{1/2+\varepsilon}
            n^{-\left(\frac{6-d}{4}\wedge1\right)+3\varepsilon}.
        \end{align}

        For $S_{22}$, \eqref{def:Q} yields
        \begin{align*}
            Q(r-s)-Q(k_n(r)-s)
            &=\int_{k_n(r)-s}^{r-s}
            \|p_v\|_{L^2}^2\,dv
            \\
            &\lesssim(r-k_n(r))(k_n(r)-s)^{-d/4}
            \\
            &\leq
            (r-k_n(r))^{\left(\frac{6-d}{4}\wedge1\right)-3\varepsilon}
            (k_n(r)-s)^{-\min\{d,2\}/4+3\varepsilon}
            \\
            &\lesssim
            n^{-\left(\frac{6-d}{4}\wedge1\right)+3\varepsilon}
            (r-s)^{-1/2+\varepsilon}.
        \end{align*}
        Here we used $r-k_n(r)\leq k_n(r)-s$ and
        $0<(\frac{6-d}{4}\wedge1)-3\varepsilon<1$.
        Moreover,
        \[
            \|\cP_{v_2}\bG-\cP_{v_1}\bG\|_\infty
            \leq\frac12\|\bG''\|_\infty|v_2-v_1|.
        \]
        Combining this with
        $L^\infty\hookrightarrow L^2
        \hookrightarrow\cB_{2,2}^{-2+d/2+2\varepsilon}$
        and the uniform smoothing bound above gives
        \begin{align}\label{est:S22}
            \|S_{22}\|_{L^p(\Omega)}
            \lesssim2^{-j\varepsilon}
            (t^*-t)^{-(2-\varepsilon)/4}
            (t-s)^{1/2+\varepsilon}
            n^{-\left(\frac{6-d}{4}\wedge1\right)+3\varepsilon}.
        \end{align}
        Combining \eqref{est:Ast-Lp2},
        \eqref{est:S21-fur}, and \eqref{est:S22},
        we conclude that
        \begin{align}\label{est:Ast-2}
            \|A_{s,t}^{(j)}(x)\|_{L^p(\Omega)}
            \lesssim2^{-j\varepsilon}
            (t^*-t)^{-(2-\varepsilon)/4}
            (t-s)^{1/2+\varepsilon}
            n^{-\left(\frac{6-d}{4}\wedge1\right)+3\varepsilon}.
        \end{align}
    \end{enumerate}

    Thus the first condition of \Cref{lem:SSL} holds with
    \[
        C_1=C2^{-j\varepsilon}
        n^{-\left(\frac{6-d}{4}\wedge1\right)+3\varepsilon},
        \qquad\varepsilon_1=\varepsilon,
        \qquad\delta_1=\frac{2-\varepsilon}{4}.
    \]
    The second condition follows from
    $\mE_s\delta A_{s,u,t}^{(j)}(x)=0$. Indeed,
    \begin{align*}
        \delta A_{s,u,t}^{(j)}(x)
        =(\mE_s-\mE_u)\int_u^t\Delta_jP_{t^*-r}^N
        \bigl[\bG(\tilde U_r)
        -\bG(U_{k_n(r)}+P_r^Nu_0)\bigr](x)\,dr,
    \end{align*}
    and the tower property applies. Hence we may take
    $C_2=0$.

    Let
    \[
        \cA_t^{(j)}(x)
        :=\int_0^t\Delta_jP_{t^*-r}^N
        \bigl[\bG(\tilde U_r)
        -\bG(U_{k_n(r)}+P_r^Nu_0)\bigr](x)\,dr.
    \]
    Its increments satisfy
    \[
        \mE_s(\cA_{s,t}^{(j)}(x)-A_{s,t}^{(j)}(x))=0,
        \qquad
        \|\cA_{s,t}^{(j)}(x)-A_{s,t}^{(j)}(x)\|_{L^p}
        \leq C_j|t-s|.
    \]
    The latter follows from boundedness of $\bG$ and the
    finite Fourier support of $\Delta_j$. The
    martingale-difference estimate along partitions
    identifies this integral with the sewing limit.
    Since
    \[
        \frac12+\varepsilon-\delta_1
        =\frac{5\varepsilon}{4}>0,
    \]
    \Cref{lem:SSL}, with deterministic time rescaling
    if needed, yields \eqref{est:step2-inter} and therefore
    \eqref{est:step2-final}.

    \begin{itemize}
        \item[\textbf{Step 3.}]
        \textbf{Uniform bounds via Kolmogorov
        continuity theorem:}
    \end{itemize}
    We use the same semigroup version of Kolmogorov's
    theorem (see also \cite[Proof of Corollary 4.6]{DKP}):
    for a continuous process starting from zero and a
    strongly continuous contraction semigroup $S$ on $V$,
    if
    \begin{align}\label{con:Kol}
        \mE\|X_t-S_{t-s}X_s\|_V^p
        \leq C_1|t-s|^{1+\alpha},\qquad\alpha>0,
    \end{align}
    then
    \begin{align}\label{est:Kol}
        \mE\sup_{t\leq T}\|X_t\|_V^p\leq C_1C_2.
    \end{align}
    Choose
    \[
        X_t=\int_0^tP_{t-r}^N
        \bigl[\bG(\tilde U_r)
        -\bG(U_{k_n(r)}+P_r^Nu_0)\bigr]\,dr.
    \]
    Since
    \[
        X_t-P_{t-s}X_s
        =\int_s^tP_{t-r}^N
        \bigl[\bG(\tilde U_r)
        -\bG(U_{k_n(r)}+P_r^Nu_0)\bigr]\,dr,
    \]
    \eqref{est:step2-final} implies \eqref{con:Kol} with
    $S=P$ on the mean-zero subspace of $L^2(\mT^d)$,
    \[
        \alpha=\frac{5p\varepsilon}{4}-1>0,
        \qquad
        C_1=Cn^{p[-(\frac{6-d}{4}\wedge1)+3\varepsilon]},
    \]
    for sufficiently large $p$. Therefore,
    \[
        \left(\mE\sup_{t\leq T}\|X_t\|_{L^2}^p
        \right)^{1/p}
        \lesssim
        n^{-\left(\frac{6-d}{4}\wedge1\right)+3\varepsilon}.
    \]
    Returning to the original $\hat I_{21}$ and using the
    deterministic estimate at the end of Step 1, we obtain
    \begin{align}\label{est:step3}
        \left(\mE\sup_{t\leq T}
        \|\hat I_{21}(t,t)\|_{L^2}^p\right)^{1/p}
        &\lesssim
        \left(\mE\sup_{t\leq T}\|X_t\|_{L^2}^p
        \right)^{1/p}+n^{-1}
        \nonumber\\
        &\lesssim
        n^{-\left(\frac{6-d}{4}\wedge1\right)+3\varepsilon}.
    \end{align}
    Smaller moments follow from Jensen's inequality.

    \begin{itemize}
        \item[\textbf{Step 4.}]
        \textbf{Buckling via Gr\"onwall Lemma:}
    \end{itemize}
    Combining \eqref{est:step3}, applied with $2p$,
    and \eqref{est:I-21-r}, we obtain
    \begin{align}\label{est:I21-sup}
        \left(\mE\sup_{t\leq T}\|I_{21}(t)\|_{L^2}^p
        \right)^{1/p}
        \lesssim
        n^{-\left(\frac{6-d}{4}\wedge1\right)+3\varepsilon}.
    \end{align}
    We now apply \Cref{lem:Gron} to \eqref{def:I2-1-2}
    with $X=\hat u^N$, $Y=\hat u^{N,n}$, $Z=I_{21}$,
    $\cS(s,t)=P_{t-s}^N$, and $\tau(s)=k_n(s)$.
    Since $\bG$ is globally Lipschitz and $P^N$ is
    $L^2$-contractive, the desired estimate \eqref{est:I2}
    follows after decreasing and relabelling
    $\varepsilon$ so that $3\varepsilon$ is replaced by
    the prescribed loss. The proof is complete.
\end{proof}

    \subsection{Estimate of $I_3$}
    \begin{lemma}
        Suppose $u_0\in  C_x^{2-\frac{d}{2}}$ and $\|\bG\|_\infty, \|\partial \bG\|_\infty<\infty$. For $ u^{N}$ from \eqref{def:sol-h-Galerkin} and $\hat u^{N}$ from \eqref{def:sol-h-G} we have
        for $p\geq1$ and for sufficiently small $\varepsilon>0$
        \begin{align}\label{est:I3}
             \left( \mE\sup_{t\in[0,T]}\|  u^{N}(t)-\hat u^{N}(t)\|_{L^2(\mT^d)}^p
             \right)^\frac{1}{p}
             \lesssim
             N^{\frac{d}{2}-2+\varepsilon}.
        \end{align}
    \end{lemma}

    \begin{proof}
    From \eqref{def:sol-h-G} and \eqref{def:sol-h-Galerkin} we have 
    \begin{align*}
        u^{N}(t)-\hat u^{N}(t)=&\Big(P_t^Nu_0+\int_0^tP_{t-s}^N\mathbf{G}( u^N(s)) d s +U^N(t)\Big)-\Big(P_t^Nu_0+\int_0^tP_{t-s}^N\mathbf{G}( \hat u^N(s)) d s +U(t)\Big)
        \\=&\int_0^tP_{t-s}^N\mathbf{G}( u^N(s)) -P_{t-s}^N\mathbf{G}( \hat u^N(s)) d s +U(t)^N-U(t).
    \end{align*}
    The claim follows by the same argument as in the proof of \Cref{lem:I1-est}.
    We omit the details.
    \end{proof}
    
    \subsection{Estimate of $I_4$}
    \begin{lemma}\label{lem:I4-no-taming}
        Suppose $u_0\in  C_x^{2-\frac{d}{2}}$ and $\|\bG\|_\infty, \|\partial \bG\|_\infty<\infty$. For $ u^{N}$ from \eqref{def:sol-h-Galerkin} and $v$ from \eqref{def:sol-h} we have
        for $\varepsilon\in \left(0,\frac{1}{2}\right)$ and $p\geq1$ \begin{align}\label{est:I4}
             \left( \mE\sup_{t\in[0,T]}\|  u^{N}(t)-v(t)\|_{L^2(\mT^d)}^p\right)^\frac{1}{p}
             \lesssim
             N^{\frac{d}{2}-2+\epsilon}
     \| u_0\|_{ C_x^{2-\frac{d}{2}}}.
        \end{align}
    \end{lemma}
    \begin{proof}   
    For each $t \in [0,T]$ we have
    \begin{align*}
        u^{N}(t)-v(t) &=
        \left( P_t-P_t^N \right)  u_0
        \ + \ U(t)-U^N(t) 
        \ + \ 
        \int _0^t  P_{t-s}^N\bG(u^N(s)) 
        -
        P_{t-s} \bG(u(s))
        ds 
        \\
        & =
        \left( P_t-P_t^N \right)  u_0
        \ + \ U(t)-U^N(t) 
        \ + \ 
        \int _0^t  
        \left( P_{t-s}^N
        -
        P_{t-s}\right) \bG(u^N(s)) 
        ds 
         \\ & \quad +  
        \int _0^t  
        P_{t-s} 
        \left(
        \bG(u^N(s)) 
        -
         \bG(u(s))
         \right)
        ds.
    \end{align*}
    By triangle inequality and applying \Cref{lem:Gron}, we derive
        \begin{align*}
      &  \left(\mE\sup_{t\in[0,T]} \left\|  u^{N}(t)-v(t) \right\|_{L^2(\mT^d)}^p\right)^\frac{1}{p}  
       \\ & \lesssim 
       \left( \mE\sup_{t\in[0,T]} \left\|  \left( P_t-P_t^N \right)  u_0 \right\|_{L^2(\mT^d)}^p\right)^\frac{1}{p} 
        +
        \left( \mE\sup_{t\in[0,T]}\left \|\int _0^t  
        \left(
        P_{t-s}^N 
        -
        P_{t-s} 
        \right)
        \bG(u^N(s))
        ds \right\|_{L^2(\mT^d)}^p\right)^\frac{1}{p} 
        \\ & 
        \quad + 
        \left( \mE\sup_{t\in[0,T]} \left\|  U(t)-U^N(t) \right\|_{L^2(\mT^d)}^p\right)^\frac{1}{p}
        =:
        I_{41}+I_{42}+I_{43}.
        \end{align*}
      It simply follows from \eqref{est:U-U^N} that 
      \begin{align}
          \label{est:I43}I_{43}\lesssim N^{\frac{d}{2}-2+\varepsilon}.
      \end{align}      
     For $I_{41}$, we use $\Pi_N$, the orthogonal projection from $L^2(\mT^d,\mC)$ to its subspace $Span\left((e_k)_{|k| \leq N} \right)$, to obtain
     \begin{align}\label{est:I41}
         I_{41} & \lesssim  {\sup_{t\in[0,T]} } \|  (P_t-P_t^N)u_0\|_{L^\infty(\mT^d)}\lesssim \| u_0-\Pi_Nu_0\|_{L^\infty(\mT^d)}\lesssim N^{\frac{d}{2}-2+\varepsilon} \|u_0\|_{C_x^1},
     \end{align}
        for $\varepsilon>0$ being sufficiently small.
     Let us now move to $I_{42}$. By the similar idea of getting \eqref{est:I41} and the result from \Cref{lem:SPDE} we have
    \begin{align}\label{est:I42}
         I_{42}  & =  
         \left( \mE\sup_{t\in[0,T]}\left \|\int _0^t  
        \left(
        P_{t-s}^N 
        -
        P_{t-s} 
        \right)
        \bG(u^N(s))
        ds \right\|_{L^2(\mT^d)}^p\right)^\frac{1}{p} 
       \nonumber \\
        & \leq
        \left( \mE\sup_{t\in[0,T]}
        \left (
        \int _0^t  
        \left \|
       \bG(u^N(s)) -\Pi_N\bG(u^N(s)) 
        \right\|_{L^\infty(\mT^d)}
        %
        ds \right)^p\right)^\frac{1}{p} 
        \nonumber   \\
        & \leq C N^{\frac{d}{2}-2+\varepsilon}\sup_{t\in[0,T]}\|u^N(t)\|_{C_x^1}\lesssim N^{\frac{d}{2}-2+\varepsilon}.
     \end{align}
    
     Collecting \eqref{est:I43}, \eqref{est:I41} and \eqref{est:I42} together shows \eqref{est:I4}.        
    \end{proof}
    \subsection{Proof of the main result}
    With estimates \eqref{est:I1},  \eqref{est:I2},  \eqref{est:I3},  \eqref{est:I4} 
    at hand, we are ready to give:
\begin{proof}[Proof of \Cref{thm:main}]
Putting the estimates \eqref{est:I1},  \eqref{est:I2},  \eqref{est:I3},  \eqref{est:I4}  together into \eqref{est:uNn-v} yields \eqref{est:main-con-taming}.
 
  The proof is finished.
\end{proof}

\section*{Acknowledgement}

Several inspiring and fruitful discussions with A. Djurdjevac, M. Gerencs\'er, H. Kremp are acknowledged. We particularly thank H. Kremp for suggesting the possible order $1$ spatial rate via assuming $C^1_x$-regular initail datum and also for correcting one estimate in Lemma 3.4. 

The authors acknowledge the support of DFG BL 535/12-1, Project number: 514726621

\end{document}